\documentclass[12pt]{amsart}

\usepackage[USenglish]{babel}
\usepackage{amsmath,amsthm,amssymb,amsfonts}
\usepackage{mathtools}
\usepackage{mathrsfs}
\usepackage{bbm}

\usepackage{indentfirst}
\usepackage{enumerate}
\usepackage{accents}
\usepackage{cancel} %cancel arrows

\usepackage{multirow}
\usepackage{longtable}
\usepackage{enumitem}

\usepackage{float}
\usepackage{graphicx}
\usepackage{placeins}
\usepackage{caption}
\usepackage{arydshln}

\usepackage{tikz}
\usetikzlibrary{patterns}
\usepackage{tikz-cd}

\usepackage[breaklinks=true, bookmarksopenlevel=1, bookmarksdepth=2]{hyperref}

\allowdisplaybreaks

\newtheorem{thm}{Theorem}[section]

\newtheorem{lem}[thm]{Lemma}
\newtheorem{prop}[thm]{Proposition}

\theoremstyle{definition}
\newtheorem{defn}[thm]{Definition}

\theoremstyle{remark}
\newtheorem{rem}[thm]{Remark}

\newtheorem*{ntt}{Notation}

\numberwithin{equation}{section}

\newcommand{\N}{\mathbb{N}}      % N = Naturals
\newcommand{\Z}{\mathbb{Z}}      % Z = Integers
\newcommand{\Q}{\mathbb{Q}}      % Q = Rationals
\newcommand{\R}{\mathbb{R}}      % R = Reals
\newcommand{\eps}{\varepsilon}   % epsilon
\newcommand\restr[2]{{           % restriction of function
  \left.\kern-\nulldelimiterspace #1%
  \right|_{#2}%
 }}

\newcommand\minus{               % small minus
  \setbox0=\hbox{-}%
  \vcenter{%
    \hrule width\wd0 height \the\fontdimen8\textfont3%
  }%
}

\newcommand{\tr}{\mathrm{tr}}    % Trace
\newcommand{\id}{\mathrm{id}}    % Identity
\newcommand{\SL}{\mathrm{SL}}    % Special Linear
\newcommand{\bigslant}[2]{       % Big Slant
  {\raisebox{.2em}{$#1$} \left/ \raisebox{-.2em}{$#2$} \right.}}

\tikzset{                        % symbol instead of arrow
    symbol/.style={%
        ,draw=none
        ,every to/.append style={%
            edge node={node [sloped, allow upside down, auto=false]{$#1$}}}
    }
}

\newcommand{\negphantom}[1]{\settowidth{\dimen0}{#1}\hspace*{-\dimen0}}

\newcommand{\Hmo}{\mathrm{Homeo}}
\newcommand{\Hmot}{\mathrm{H}\widetilde{\mathrm{omeo}}}

\newcommand{\SLt}{\widetilde{\mathrm{SL}}}
\newcommand{\PSL}{\mathrm{PSL}}

\newcommand{\bbS}{\mathbb{S}}

\newcommand{\Cl}{\mathrm{C}\ell}

\restylefloat{table}

\setlist[itemize]{leftmargin=*}
\setlist[enumerate]{leftmargin=*}

\makeatletter
\def\@tocline#1#2#3#4#5#6#7{\relax
  \ifnum #1>\c@tocdepth % then omit
  \else
    \par \addpenalty\@secpenalty\addvspace{#2}%
    \begingroup \hyphenpenalty\@M
    \@ifempty{#4}{%
      \@tempdima\csname r@tocindent\number#1\endcsname\relax
    }{%
      \@tempdima#4\relax
    }%
    \parindent\z@ \leftskip#3\relax \advance\leftskip\@tempdima\relax
    \rightskip\@pnumwidth plus4em \parfillskip-\@pnumwidth
    #5\leavevmode\hskip-\@tempdima
      \ifcase #1
       \or\or \hskip 1em \or \hskip 2em \else \hskip 3em \fi%
      #6\nobreak\relax
    \dotfill\hbox to\@pnumwidth{\@tocpagenum{#7}}\par
    \nobreak
    \endgroup
  \fi}
\makeatother

\begin{document}

\title{Classification of the conjugacy classes of \texorpdfstring{$\widetilde{\mathrm{SL}}_2(\mathbb{R})$}{SL2\~{}(R)}}%
\author{Christian T\'afula}%
\address{D\'epartment de Math\'ematiques et Statistique, %
 Universit\'e de Montr\'eal, %
 CP 6128 succ Centre-Ville, %
 Montreal, QC H3C 3J7, Canada}%
\email{christian.tafula.santos@umontreal.ca}%


% ----------------------------------------------------------------
\begin{abstract}
 In this note, we classify the conjugacy classes of $\widetilde{\mathrm{SL}}_2(\mathbb{R})$, the universal covering group of $\mathrm{PSL}_2(\mathbb{R})$. For any non-central element $\alpha \in \widetilde{\mathrm{SL}}_2(\mathbb{R})$, we show that its conjugacy class may be determined by three invariants:
 \begin{enumerate}[label=(\roman*)]
  \item \underline{Trace}: the trace (valued in the set of positive real numbers $\mathbb{R}_{+}$) of its image $\overline{\alpha}$ in $\mathrm{PSL}_2(\mathbb{R})$;

  \item \underline{Direction type}: the sign behavior of the induced self-homeomorphism of $\mathbb{R}$ determined by the lifting $\widetilde{\mathrm{SL}}_2(\mathbb{R}) \curvearrowright \R$ of the action $\mathrm{PSL}_2(\mathbb{R}) \curvearrowright \mathbb{S}^{1}$;

  \item \underline{The function $\ell^{\sharp}$}: a conjugacy invariant length function introduced by S. Mochizuki [Res. Math. Sci. {\bf{3}} (2016), 3:6].
 \end{enumerate}
\end{abstract}
\maketitle
% ----------------------------------------------------------------
%\tableofcontents

\section{Introduction}
 Our goal in this note is to classify the conjugacy classes of the universal covering group of $\PSL_2(\R)$, namely, $\SLt_2(\R)$ --- one of the eight Thurston model geometries that appear in the geometrization of $3$-manifolds. Based on the group-theoretic proof of Szpiro's theorem as described by Zhang \cite{zhan01} (originally due to Bogomolov et al. \cite{amobogkatpan00}), we start by identifying $\PSL_2(\R)$ with a subgroup of $\Hmo_{+}(\bbS^1)$, the group of orientation-preserving self-homeomorphisms of $\bbS^1$. This is done via the faithful continuous action induced by $\PSL_2(\R)$ on $\bbS^1$. Then, realizing the universal covering group $\Hmot_{+}(\bbS^1)$ as a subgroup of the group of increasing self-homeomorphisms $\Hmo_{+}(\R)$ of $\R$, we obtain the central extension
 \begin{equation*}
  \begin{tikzcd}
   1\arrow[r] & \Z\arrow[r] & \Hmot_{+}(\bbS^1)\arrow[r] & \Hmo_{+}(\bbS^1)\arrow[r] & 1,
  \end{tikzcd}
 \end{equation*}
 which allows us to construct $\SLt_2(\R)$ by lifting $\PSL_2(\R)$ into the larger group. We show that the conjugacy class of an element $\alpha \in \SLt_2(\R)$ is determined uniquely by some $\mathcal{T}_n \in Z\big(\Hmot_{+}(\bbS^1)\big)$ together with the \emph{canonical lifting} (as defined in Subsection \ref{canoli}) of a certain representative of the conjugacy class of $\overline{\alpha} \in \PSL_2(\R)$.
 
%%%%%%%%%%%%%%%%%%%%%%
\subsection{Conjugacy classes of \texorpdfstring{$\PSL_2(\R)$}{PSL2(R)}}\label{introtbl}
 By Iwasawa decomposition, every $\begin{bsmallmatrix} a & b \\ c & d \end{bsmallmatrix} \in \PSL_2(\R)$ has a unique representation as $\begin{bsmallmatrix} a & b \\ c & d \end{bsmallmatrix} = \rho_{\vartheta}\mathrm{a}_{\lambda}\mathrm{u}_{x}$, where $\rho_{\vartheta} \in K$, $\mathrm{a}_{\lambda}\in A$, and $\mathrm{u}_{x}\in N$, where
 \begin{align*}
  K &:= \left\{ \rho_{\vartheta} := \begin{bmatrix} \cos(\vartheta) & -\sin(\vartheta) \\ \sin(\vartheta) & \cos(\vartheta) \end{bmatrix} ~\big|~ 0 \leq \vartheta < \pi \right\}, \\
  A &:= \left\{ \mathrm{a}_{\lambda} := \begin{bmatrix} \lambda & 0 \\ 0 & 1/\lambda \end{bmatrix} ~\big|~ \lambda > 0 \right\}, \\
  N &:= \left\{ \mathrm{u}_{x} := \begin{bmatrix} 1 & x \\ 0 & 1 \end{bmatrix} ~\big|~ x \in \R \right\},
 \end{align*}
 Moreover, write $\overline{\tr}(\begin{bsmallmatrix} a & b \\ c & d \end{bsmallmatrix}) := |a+d|$ for $\begin{bsmallmatrix} a & b \\ c & d \end{bsmallmatrix} \in\PSL_2(\R)$, and $\overline{\mathrm{Id}} := \begin{bsmallmatrix} 1 & 0 \\ 0 & 1 \end{bsmallmatrix}$. Then:
 \begin{itemize}
  \item If $\overline{\tr}(\begin{bsmallmatrix} a & b \\ c & d \end{bsmallmatrix}) > 2$, then there is a unique $\lambda > 1$ for which $\begin{bsmallmatrix} a & b \\ c & d \end{bsmallmatrix} \sim \mathrm{a}_{\lambda}$;
  \item If $\overline{\tr}(\begin{bsmallmatrix} a & b \\ c & d \end{bsmallmatrix}) = 2$, then $\begin{bsmallmatrix} a & b \\ c & d \end{bsmallmatrix}$ is conjugate to either $\overline{\mathrm{Id}}$, $\mathrm{u}_1$ or $\mathrm{u}_{-1}$;
  \item If $\overline{\tr}(\begin{bsmallmatrix} a & b \\ c & d \end{bsmallmatrix}) < 2$, then there is a unique $0 < \vartheta < \pi$ for which $\begin{bsmallmatrix} a & b \\ c & d \end{bsmallmatrix} \sim \rho_{\vartheta}$.
 \end{itemize}
 
 Denoting by $\widetilde{\rho}_{\vartheta}$, $\widetilde{\mathrm{a}}_{\lambda}$, $\widetilde{\mathrm{u}}_x$ the canonical liftings of $\rho_{\vartheta}$, $\mathrm{a}_{\lambda}$, $\mathrm{u}_x$ in the sense of Subsection \ref{canoli}, we will show the following:
 
 \noindent
 \renewcommand{\arraystretch}{1.5}
 \begin{longtable}[c]{|c|c|c||c|}
  \cline{1-4}
  \multirow{2}{*}{\textnormal{\textbf{Trace}}} & \multirow{2}{*}{$\begin{array}{@{}c@{}} \textbf{Direction} \\[-.5em] \textbf{type} \end{array}$} & \multirow{2}{*}{$\ell^{\sharp}$} & \multirow{2}{*}{$\begin{array}{@{}c@{}} \textbf{Conjugacy} \\[-.1em] \textbf{class in } \SLt_2(\R) \end{array}$} \\*
  & & & \\
  \hline \endhead \hline\hline
  \multirow{4}{*}{$\overline{\tr} < 2$} & \multirow{2}{*}{Forward} & $[(0,1)]$ & $[\widetilde{\rho}_{\vartheta}]$ \\* \cline{3-4}
  & & $[(n,n+1)]$ & $[\mathcal{T}_{n}\,\widetilde{\rho}_{\vartheta}]$ \\* \cline{2-4}
  & \multirow{2}{*}{Backward} & $[(0,1)]$ & $[\mathcal{T}_{-1}\,\widetilde{\rho}_{\vartheta}]$ \\ \cline{3-4}
  & & $[(n,n+1)]$ & $[\mathcal{T}_{-1-n}\,\widetilde{\rho}_{\vartheta}]$ \\ \cline{2-4}
  \hline\hline %%%

  \multirow{6}{*}{$\begin{array}{@{}c@{}} \overline{\tr} = 2 \\ (\overline{\alpha} \neq \id_{\bbS^1}) \end{array}$} & \multirow{2}{*}{Forward} & $[n]$ & $[\mathcal{T}_{n}\,\widetilde{\mathrm{u}}_1]$ \\* \cline{3-4}
  & & $[(n,n+1)]$ & $[\mathcal{T}_{n}\,\widetilde{\mathrm{u}}_{-1}]$ \\* \cline{2-4}
  & Semi-forward & $[(0,1)]$ & $[\widetilde{\mathrm{u}}_{-1}]$ \\* \cline{2-4}
  & Semi-backward & $[(0,1)]$ & $[\widetilde{\mathrm{u}}_1]$ \\* \cline{2-4}
  & \multirow{2}{*}{Backward} & $[n]$ & $[\mathcal{T}_{-n}\, \widetilde{\mathrm{u}}_{-1}]$ \\* \cline{3-4}
  & & $[(n,n+1)]$ & $[\mathcal{T}_{-n}\,\widetilde{\mathrm{u}}_1]$ \\*
  \hline\hdashline

  \multicolumn{4}{:c:}{$\begin{array}{c:c:c::c}
  \multirow{3}{*}{$\hspace{-.635em} \begin{array}{@{}c@{}} \overline{\alpha}= \id_{\bbS^1} \\ \textnormal{(}\alpha \textit{ central}\textnormal{)} \end{array} ~\,\hspace{-.635em}$} & \text{Forward} & [n] & \ [\mathcal{T}_n] \\ \cdashline{2-4}
  & \hspace{1.6em}\textit{Identity}\hspace{1.6em} & \hspace{1.835em}[0]\hspace{1.835em} & \hspace{2.19em}\ [\id_{\R}]\hspace{2.19em} \\ \cdashline{2-4}
  & \text{Backward} & [n] & \ [\mathcal{T}_{-n}] \end{array}$} \\
  \hdashline\hline

  \multirow{3}{*}{$\overline{\tr} > 2$} & Forward & $[(n,n+1)]$ & $[\mathcal{T}_{n}\, \widetilde{\mathrm{a}}_{\lambda}]$ \\* \cline{2-4}
  & Alternating & $[(0,1)]$ & $[\widetilde{\mathrm{a}}_{\lambda}]$ \\* \cline{2-4}
  & Backward & $[(n,n+1)]$ & $[\mathcal{T}_{-n}\, \widetilde{\mathrm{a}}_{\lambda}]$ \\*
  \cline{1-4}
 \end{longtable}

 The three invariants above may be briefly described as follows:
 
 \begin{enumerate}[label=(\roman*)]
  \item \underline{Trace}: $\overline{\tr}(\overline{\alpha}) := |a+d|$, where $\begin{bsmallmatrix} a & b \\ c & d \end{bsmallmatrix} = \overline{\alpha} \in \PSL_2(\R)$ is the image of $\alpha \in\SLt_2(\R)$ via the natural quotient map; \medskip

  \item \underline{Direction type}: Elements $\alpha \in \Hmot_{+}(\bbS^1)$ will be constructed in such a way that the function $\R\ni x\mapsto\alpha(x)-x\in\R$ is invariant with respect to the translation $\R\ni x \mapsto x+1\in\R$. The sign behavior of $\alpha(x)-x$ can be classified as follows: strictly positive (the \emph{forward} direction type), strictly negative (the \emph{backward} direction type), positive but not strictly positive (the \emph{semi-forward} direction type), negative but not strictly negative (the \emph{semi-backward} direction type), or alternating in sign (the \emph{alternating} direction type). %These five possibilities are called \emph{direction types}.
  \medskip

  \item \underline{The function $\ell^{\sharp}$}: As was mentioned above, $\alpha(x)-x$ is invariant under translation by $1$; hence, in order to determine the behavior of $\alpha$, it suffices to restrict our attention to $x\in [0,1]$, which is compact.  In particular, $\ell(\alpha) := \sup_{x\in[0,1]} |\alpha(x) - x|$ is  a well-defined real number. We then define the \emph{$\ell^{\sharp}$} function by taking $\ell^{\sharp}(\alpha) := \{\lfloor \ell(\alpha)\rfloor, \lceil \ell(\alpha) \rceil\} \subseteq \N$.
 \end{enumerate}

 \noindent
 The direction type and $\ell^{\sharp}$ function are invariant under conjugation in $\Hmot_{+}(\bbS^1)$ (Lemma \ref{dirinv} and Proposition \ref{lcinv}, respectively). The $\ell^{\sharp}$ function, which we describe in Subsection \ref{lsharp}, is based on the exposition of Bogomolov's proof of Szpiro's theorem by Mochizuki \cite{moc16}.
 
% %%%%%%%%%%%%%%%%
\begin{ntt}
 Let $G$ be a group. Given $g_1,g_2\in G$, we say that $g_1$ is \emph{conjugate} to $g_2$ if there exists $h\in G$ such that $g_1 = h g_2 h^{-1}$, in which case we write $g_1 \sim g_2$. Conjugacy is an equivalence relation; we denote the \emph{conjugacy class} of $g\in G$ by $[g] := \{ h\in G ~|~ g\sim h \}$ and the set of all conjugacy classes of $G$ by $\Cl(G)$. For any two elements $g,h\in G$, we denote their \emph{commutator} by $[g,h] := ghg^{-1}h^{-1}$.
\end{ntt}

\section{Realizing \texorpdfstring{$\SLt_2(\R)$}{SL2\~{}(R)}}
 
 We start by briefly recalling the construction of $\SLt_2(\R)$ via the faithful continuous action $\PSL_2(\mathbb{R}) \curvearrowright \bbS^1$. For a deeper exposition, cf. Ghys \cite{ghy01}.
 
%%%%%%%%%%%%
\subsection{\texorpdfstring{$\PSL_2(\R)$}{PSL2(R)} action on the circle}\label{sec2}
 Consider the following basic lemma in topology:

 \begin{lem}\label{ymlm}
  Let $X,Y,Z$ be topological spaces, with $Y$ locally compact Hausdorff, and let $S \subseteq \mathcal{C}(Y,Z)$ be a set of continuous maps endowed with the compact-open topology. Given a map $\phi:X\to S$, the following are equivalent:
  \begin{enumerate}[label=\textnormal{(\roman*)}]
   \item $\phi:X \ni x \mapsto \phi_x \in S$ is continuous;

   \item The map $\Phi: X\times Y \ni (x,y) \mapsto \phi_{x}(y) \in Z$ is continuous.
  \end{enumerate}
 \end{lem}
%  \begin{proof}
%   The topology on $S$ is generated by the sets $V_{K,U} := \{f\in S ~|~ f(K)\subseteq U\}$, for every $K\subseteq Y$ compact and $U\subseteq Z$ open, hence, the map $\phi$ is continuous if and only if the sets $\mathcal{V}_{K,U} := \{x\in X ~|~ \phi_{x}(K) \subseteq U \} \subseteq X$ are open.
% 
%   Now, consider $\Phi^{-1}(U) = \{(x,y) ~|~ \phi_{x}(y) \in U \} \subseteq X\times Y$, and take $(x,y)\in \Phi^{-1}(U)$. Since $\phi_x$ is continuous and $Y$ is locally compact Hausdorff, $y$ has an open neighborhood $N_y$ such that its closure $\overline{N}_y$ is compact and $y\in N_y \subseteq \overline{N}_y \subseteq \phi_x^{-1}(U)$. Thus, $(x,y) \in \mathcal{V}_{\overline{N}_{y}, U} \times N_{y} \subseteq \Phi^{-1}(U)$, from where it is clear that $\text{(i)}\implies\text{(ii)}$.
% 
%   For the other direction, fix $\mathcal{V}_{K,U} \subseteq X$ and take $x \in \mathcal{V}_{K,U}$. If $\Phi$ is continuous, then for every $y\in K$ there are open neighborhoods $W^{(y)}_x \ni x$ (in $X$) and $N_y \ni y$ (in $Y$) such that $W^{(y)}_x \times N_y \subseteq \Phi^{-1}(U)$. Since $\{N_y ~|~ y\in K\}$ covers $K$, which is compact, we may take a finite subcovering $N_{y_1},\ldots, N_{y_n}$, with $n\in \N$ and $y_i \in K$, that covers $K$ as well. Thus, take $W_x := \bigcap_{i=1}^{n} W^{(y_i)}_{x}$, which is a finite intersection of open sets, hence open. It follows that $x \in W_x \subseteq \mathcal{V}_{K,U}$, therefore $\mathcal{V}_{K,U}$ is open and $\text{(ii)}\implies\text{(i)}$.
%  \end{proof}

 In our case, the natural action of $\PSL_2(\R)$ on $\mathbb{P}_{\R}^1$ is given by homographies -- i.e., M\"obius transformations
 \[ \PSL_{2}(\R) \times \mathbb{P}_{\R}^1 \ni \left(\begin{bmatrix} a & b \\ c & d\end{bmatrix}, (z:1)\right) \longmapsto \bigg(\frac{az+b}{cz+d} : 1\bigg) \in \mathbb{P}_{\R}^1, \]
 which are continuous. The real projective line can then be identified with $\R/\Z\simeq \bbS^1$. Thus, Lemma \ref{ymlm} gives us continuous map $\PSL_2(\R) \rightarrow \Hmo_{+}(\bbS^1)$, which amounts to a continuous action $\PSL_2(\R) \curvearrowright \bbS^1$. The fact that this action is faithful follows from the fact that if $M,N\in \SL_2(\R)$ maps 1-dimensional subspaces to the same 1-dimensional subspace, then $MN^{-1}$ is a homothety from the origin, i.e., of the form $\begin{psmallmatrix} \lambda & 0 \\ 0 & \lambda \end{psmallmatrix}$ for some $\lambda \in \R$. Since $\det(MN^{-1}) = 1$, this implies $\lambda = \pm 1$, and thus $\overline{M} = \overline{N} \in \PSL_2(\R)$. It follows that the map $\PSL_2(\R) \hookrightarrow \Hmo_{+}(\bbS^1)$ is an embedding of topological groups.
 
 \FloatBarrier
 \begin{figure}[!htb]
 \centering
 \begin{minipage}{0.32\textwidth}\begin{center}
  \begin{tikzpicture}[scale=0.4, every node/.style={scale=0.4}]
%   \draw[help lines, black!20] (-5,-5) grid (5,5);
   \draw[black] (0,-0.02) node {\large{$\bullet$}};
   \draw[black] (4.5,-4.5) node {\resizebox{1cm}{.8cm}{\Huge{$K$}}};

   \draw[->, dashed, black] (0,-5)--(0,5);
   \draw[->, dashed, black] (-5,0)--(5,0);

   \draw[black] (0,0) circle (4cm);
   \draw[black] (4,-0.02) node {\Huge{$\circ$}};

   \draw[->, thick, black] (4.48, 0.4) to[out=95,in=-32.5] (2.42, 3.80);
   \draw[->, thick, black] (2.08, 3.99) to[out=152.5,in=27.5] (-2.08, 3.99);
   \draw[->, thick, black] (-2.42, 3.80) to[out=212.5,in=85] (-4.48, 0.4);
   \draw[->, thick, black] (-4.48, -0.4) to[out=-85,in=-212.5] (-2.42, -3.80);
   \draw[->, thick, black] (-2.08, -3.99) to[out=-27.5,in=-152.5] (2.08, -3.99);
   \draw[->, thick, black] (2.42, -3.80) to[out=32.5,in=-95] (4.48, -0.4);
  \end{tikzpicture}
 \end{center}\end{minipage}
 \begin{minipage}{0.32\textwidth}\begin{center}
  \begin{tikzpicture}[scale=0.4, every node/.style={scale=0.4}]
%   \draw[help lines, black!20] (-5,-5) grid (5,5);
   \draw[black] (0,-0.02) node {\large{$\bullet$}};
   \draw[black] (4.5,-4.5) node {\resizebox{.9cm}{.8cm}{\Huge{$A$}}};

   \draw[->, dashed, black] (0,-5)--(0,5);
   \draw[->, dashed, black] (-5.5,0)--(5.5,0);

   \draw[black] (0,0) circle (4cm);
   \draw[black] (4,-0.02) node {\Huge{$\circ$}};

   \draw[thick, black] (4.5,0) node {\resizebox{1.5cm}{1.2cm}{\Huge{$\mathbf{\times}$}}};
   \draw[thick, black] (-4.5,0) node {\resizebox{1.5cm}{1.2cm}{\Huge{$\mathbf{\times}$}}};

   \draw[->, thick, black] (3.99, 2.08) to[out=-62.5,in=95] (4.48, 0.4);
   \draw[->, thick, black] (-2.08, 3.99) to[out=27.5,in=122.5] (3.80, 2.42);
   \draw[->, thick, black] (-4.48, 0.4) to[out=85,in=212.5] (-2.42, 3.80);
   \draw[->, thick, black] (-4.48, -0.4) to[out=-85,in=-212.5] (-2.42, -3.80);
   \draw[->, thick, black] (-2.08, -3.99) to[out=-27.5,in=-122.5] (3.80, -2.42);
   \draw[->, thick, black] (3.99, -2.08) to[out=62.5,in=-95] (4.48, -0.4);
  \end{tikzpicture}
 \end{center}\end{minipage}
 \begin{minipage}{0.32\textwidth}\begin{center}
  \begin{tikzpicture}[scale=0.4, every node/.style={scale=0.4}]
%   \draw[help lines, black!20] (-5,-5) grid (5,5);
   \draw[black] (0,-0.02) node {\large{$\bullet$}};
   \draw[black] (4.5,-4.5) node {\resizebox{1cm}{.8cm}{\Huge{$N$}}};

   \draw[->, dashed, black] (0,-5)--(0,5);
   \draw[->, dashed, black] (-5,0)--(5.5,0);

   \draw[black] (0,0) circle (4cm);
   \draw[black] (4,-0.02) node {\Huge{$\circ$}};

   \draw[thick, black] (4.5,0) node {\resizebox{1.5cm}{1.2cm}{\Huge{$\mathbf{\times}$}}};

   \draw[->, thick, black] (3.99, 2.08) to[out=-62.5,in=95] (4.48, 0.4);
   \draw[->, thick, black] (0.4, 4.48) to[out=-5,in=122.5] (3.80, 2.42);
   \draw[->, thick, black] (-4.48, 0.4) to[out=85,in=185] (-0.4, 4.48);
   \draw[->, thick, black] (-0.4, -4.48) to[out=-185,in=-85] (-4.48, -0.4);
   \draw[->, thick, black] (3.80, -2.42) to[out=-122.5,in=5] (0.4, -4.48);
   \draw[->, thick, black] (4.48, -0.4) to[out=-95,in=62.5] (3.99, -2.08);
  \end{tikzpicture}
 \end{center}\end{minipage}
  \captionsetup{singlelinecheck=off}
  \caption{Illustration of how generic elements from each of the groups from the Iwasawa decomposition of $\PSL_2(\R)$ act on $\bbS^1$. The circles represent $\bbS^1$ in the usual counter-clockwise orientation; the arrows indicate, roughly, the direction and length of the rate of change in the sector below it; and ``$\times$'' indicates fixed points.}
 \label{fig3}
 \end{figure}

%%%%%%%%%%%%%%%%%%%% 
\subsection{Universal covering of \texorpdfstring{$\Hmo_{+}(\bbS^1)$}{Homeo\_+(S\^{}1)}}
 This group may be described as
 \begin{gather*}
  \Hmot_{+}(\bbS^1) := \big\{ \alpha\in \Hmo(\R) ~\big|~ \alpha(x+1) = \alpha(x)+1 \text{ for all } x\in\R \big\}.
 \end{gather*}
 Each $\alpha \in \Hmot_{+}(\bbS^1)$ defines a self-homeomorphism $\overline{\alpha} \in \Hmo_{+}(\bbS^1)$ via $\overline{\alpha}(\overline{x}) := \overline{\alpha(x)}$. The simplest type of elements in $\Hmot_{+}(\bbS^1)$ are the \emph{translations} $\mathcal{T}_u$ ($u\in \R$), defined by
 \begin{equation}
  \forall x \in \R,~ \mathcal{T}_{u}(x) := x + u. \label{dftransl}
 \end{equation}
 Since the kernel of $\Hmot_{+}(\bbS^1) \ni \alpha \mapsto \overline{\alpha} \in \Hmo_{+}(\bbS^1)$ consists of the subgroup $\{\mathcal{T}_{n}\}_{k \in \Z}$, which is discrete, this map is indeed a covering morphism. Moreover:
 
 \begin{lem}\label{auttcent}
  $Z\big(\Hmot_{+}(\bbS^1)\big) = \{\mathcal{T}_n\}_{n\in \Z}$ \textnormal{($\simeq \Z$)}.
 \end{lem}
 
 We include a short proof for the sake of completeness.
 
 \begin{proof}
  We will show if $\zeta \in Z\big(\Hmot_{+}(\bbS^1)\big)$ then $\zeta = \mathcal{T}_M$ for some $M\in \Z$. For each $q \in \Z_{\geq 1}$ and $x \in \R$, let
  \begin{equation*}
   \alpha_{q}(x) := x + \frac{|\sin(\pi q x)|}{\pi q}.
  \end{equation*}
  One checks that $\alpha_q$ is strictly increasing, since its first (left or right) derivative is never negative and has a discrete zero set; moreover, $\alpha_q(x+1) = \alpha_q(x)+1$ for all $x\in\R$, and so $\alpha_q \in \Hmot_{+}(\bbS^1)$. Since $\zeta$ is in the center, we have
  \begin{align*}
   \zeta(\alpha_1(0)) = \alpha_1(\zeta(0)) &\implies \zeta(0) = \zeta(0) + \frac{|\sin(\pi \zeta(0))|}{\pi} \\
   &\implies |\sin(\pi \zeta(0))| = 0 \\
   &\implies \zeta(0) \in \Z.
  \end{align*}

  We claim that $\zeta = \mathcal{T}_{M}$ for $M := \zeta(0)$. To this end, let $\xi(x) := \zeta(x) - M$. It suffices to show that $\xi$ is the identity when restricted to $\Q$, for $\Q$ is dense in $\R$ and $\xi$ is continuous. Fixing $q\geq 1$, for every $p \in \{0, 1,\ldots, q-1\}$ we have
  \begin{align*}
   \alpha_q\xi(p/q) = \xi\alpha_q(p/q) &\implies  \xi(p/q) + \frac{|\sin(\pi q\, \xi(p/q))|}{\pi q} = \xi\left(p/q + \frac{|\sin(\pi q \, (p/q))|}{\pi q} \right) \\
   &\implies |\sin(\pi q\, \xi(p/q))| = 0 \\
   &\implies q\, \xi(p/q) \in \Z.
  \end{align*}
  Since $\xi$ is strictly increasing, we have $0 \leq \xi(p/q) < 1$ for all $0\leq p < q \in \Z_{\geq 1}$, and thus
  \begin{equation*}
   0 < q\, \xi(1/q) < q\, \xi(2/q) < \cdots < q\, \xi(1-1/q) < q \in \Z,
  \end{equation*}
  which implies $q\, \xi(p/q) = p$. Since $\mathcal{T}_{1}\xi = \xi\mathcal{T}_{1}$, it follows that $\xi(r) = r$ for all $r\in \Q$, concluding the proof.
 \end{proof}

 The group $\SLt_2(\R)$ may then be characterized as the group making the following diagram commute, with exact rows:
 \begin{equation}
  \begin{tikzcd}[column sep=scriptsize, row sep=scriptsize]
   1\arrow[r] & \Z\arrow[d, symbol=\simeq]\arrow[r] & \SLt_2(\R) \arrow[dd, hook]\arrow[r]\arrow[ddr, phantom, "\ulcorner", very near start] & \PSL_2(\R) \arrow[dd, hook]\arrow[r] & 1\phantom{.} \\[-.5em]
    & Z\big(\Hmot_{+}(\bbS^1)\big)\arrow[d, symbol=\text{$=$}] & & & \\[-.5em]
   1\arrow[r] & \{\mathcal{T}_{n}\}_{n\in\Z} \arrow[r] & \Hmot_{+}(\bbS^1)\arrow[r] & \Hmo_{+}(\bbS^1)\arrow[r] & 1.
  \end{tikzcd} \label{gtilde}
 \end{equation}
 
 \begin{rem}[Alternative construction]
  In 1964, L. Puk\'{a}nszky \cite[Section 1.B, \textbf{a.}]{puk64} constructed $\SLt_{2}(\R)$ by lifting the elements $\rho_{\vartheta}$, $\mathrm{a}_{\lambda}$, $\mathrm{u}_{x}$ in the Iwasawa decomposition of $\PSL_2(\R)$; i.e., without recourse to $\Hmot_{+}(\bbS^1)$.
 \end{rem}

 %%%%%%%%%%%%%%%%%%%%%%%%%%%%%
\subsection{Canonical liftings}\label{canoli}
 Consider the unique family of liftings $\widetilde{\varphi}\in \Hmot_{+}(\bbS^1)$ of $\varphi\in\Hmo_{+}(\bbS^1)$ defined by
 \begin{gather*}
  \widetilde{\varphi}(0) := \text{the unique real number in } [0,1) \text{ such that } \overline{\widetilde{\varphi}(0)} = \varphi(\overline{0})
 \end{gather*}
 and, for $x\in [0,1)$,
 \begin{equation*}
  \widetilde{\varphi}(x) := \text{the unique lifting of } \varphi(\overline{x}) \text{ in } [\widetilde{\varphi}(0), \widetilde{\varphi}(0) + 1).
 \end{equation*}
 For any $n\in\Z$, $\widetilde{\varphi}(x+n) := \widetilde{\varphi}(x)+n$. With the following two assertions, we call these the \emph{canonical liftings}.

  \medskip
  \noindent
  $\bullet \text{ \underline{Assertion 1}:}$ $\widetilde{\varphi} \in \Hmot_{+}(\bbS^1)$.

  For any $0\leq t < 1$, we have $\varphi\big([\overline{0},\overline{t})\big) = [\varphi(\overline{0}),\varphi(\overline{t}))$, hence $\widetilde{\varphi}\big([0,t)\big) = [\widetilde{\varphi}(0), \widetilde{\varphi}(t))$, implying that $\restr{\widetilde{\varphi}}{[0,1)}$ is bijective, bi-continuous, and increasing. Since
  \[ \widetilde{\varphi}(x) = \lfloor x\rfloor + \restr{\widetilde{\varphi}}{[0,1)}(x-\lfloor x\rfloor) \]
  for all $x\in\R$, it is clear that $\widetilde{\varphi}$ is also bijective and bi-continuous.

  \medskip
  \noindent
  $\bullet \text{ \underline{Assertion 2}:}$ If $\overline{\alpha} = \varphi$, then $\alpha = \mathcal{T}_n\, \widetilde{\varphi}$ for some $n\in\Z$.

  For $0\leq t < 1$, we have $\overline{\alpha(t)} = \overline{\alpha}(\overline{t}) = \varphi(\overline{t})$; thus, by construction, $\alpha(x)-\widetilde{\varphi}(x)$ is an integer, for every $x\in \R$. Since both $\alpha$ and $\widetilde{\varphi}$ are continuous and $\Z$ is discrete, we conclude that $\alpha(x)-\widetilde{\varphi}(x)$ is constant, and thus $\alpha \widetilde{\varphi}^{-1} \in \{\mathcal{T}_n\}_{n\in\Z}$.

 \FloatBarrier
 \begin{figure}[!htb]
  \centering
  \begin{tikzpicture}
   \foreach \t in {-2,...,2}
   \draw[-, dotted] (3*\t,3.5)--(3*\t,-0.5);

   \fill[pattern=north east lines, pattern color=black!50] (0,3) -- (1.5,0) -- (4.5,0) -- (3,3);

   \draw[->] (-6.5,3)--(6.75,3);
   \draw[->] (-6.5,0)--(6.75,0);

   \draw (0,3) node {\footnotesize{$\bullet$}};
   \draw (1.5,0) node {\footnotesize{$\bullet$}};

   \foreach \stp in {-2,...,2}
   \draw (3*\stp,3) node {\footnotesize{$\circ$}};
   \foreach \stpp in {-2,...,1}
   \draw (3*\stpp +1.5,0) node {\footnotesize{$\circ$}};

   \foreach \st in {-2,...,1}
   \draw[->, black] (3*\st,3) -- (3*\st + 1.5,0);
   \foreach \stt in {-5,-4,-2,-1,1,2,4,5}
   \draw[->, dashed, black] (\stt,3) -- (\stt + 1.5,0);

   %\draw[->, black] (6,3) to[out=-90,in=90] (7.5,0);
   %\draw[->, black] (-7,3) to[out=-90,in=90] (-5.5,0);

   \draw[-, black] (6,3) -- (6.75,1.5);
   \draw[->, dashed, black] (-6.5,1.79) -- (-5.5,0);

   \draw[thick] (7,2.5) node[anchor=south west] {$\R$};
   \draw[thick] (7,0.5) node[anchor=north west] {$\R$};
   \draw[->] (7.3,2.3)--(7.3,0.7);
   \draw[thick] (7.4,1.8) node[anchor=north west] {$\widetilde{\varphi}$};

   \foreach \x in {-2,...,2}
   \draw (3*\x,3)--(3*\x,3.1) node[anchor=south east] {\x};
   \foreach \y in {-2,...,2}
   \draw (3*\y,0)--(3*\y,-0.1) node[anchor=north east] {\y};

   \draw[anchor=north, thick] (1.8,-0.1) node {\rotatebox{-15}{$\widetilde{\varphi}(0)$}};
   \draw[anchor=north, thick] (4.8,-0.1) node {\rotatebox{-15}{$\widetilde{\varphi}(1)$}};
  \end{tikzpicture}
  \vspace{-1em}
  \captionsetup{singlelinecheck=off}%, width=\textwidth}
  \caption{Canonical lifting of $\varphi: \bbS^1 \ni \overline{x} \,\mapsto\, \overline{x + 1/2} \in \bbS^1$.}
  \label{fig1}
 \end{figure}

 \begin{rem}[Alternative liftings]\label{zhrem}
  In Lemma 3.5, p. 150 of Zhang \cite{zhan01}, a different section of the covering $\SLt_2(\R) \twoheadrightarrow \PSL_2(\R)$ is considered. The construction of such alternative liftings may be briefly reproduced as follows. For an element $\gamma \in \PSL_2(\R)$, there are two possibilities:
  \begin{enumerate}[label=(\roman*)]
   \item \emph{\underline{Fixed point}:} If $\gamma$ has a real eigenspace,\footnote{Meaning that, if $\gamma = \begin{bsmallmatrix} a & b \\ c & d \end{bsmallmatrix}$, then either $\begin{psmallmatrix} a & b \\ c & d \end{psmallmatrix}$ or $\begin{psmallmatrix} -a & -b \\ -c & -d \end{psmallmatrix} \in \SL_2(\R)$ has a real positive eigenvalue.} then its induced self-homomorphism $\varphi_{\gamma} \in \Hmo_{+}(\bbS^1)$ has a fixed point. Thus, one takes $\widehat{\gamma} \in \SLt_2(\R)$ to be the \emph{unique} lifting of $\varphi_{\gamma}$ which has a fixed point.

   \medskip
   \item \emph{\underline{General case}:} Let $\gamma = \begin{bsmallmatrix} a & b \\ c & d \end{bsmallmatrix}$. If $c=0$, then $(1,0)^{\mathrm{T}} \in \R^{2}$ is an eigenvector of $\begin{psmallmatrix} a & b \\ c & d \end{psmallmatrix} \in \SL_2(\R)$, so it falls into the previous case. We may assume then that $c\neq 0$. With that, we have
   \[ \det\left(\begin{pmatrix} a & b \\ c & d \end{pmatrix} - \begin{pmatrix} 1 & x \\ 0 & 1 \end{pmatrix}\right) = 0, ~\text{ where } x := b - \frac{(a-1)(d-1)}{c}, \]
   and thus, $\det(\begin{psmallmatrix} a & b \\ c & d \end{psmallmatrix}\begin{psmallmatrix} 1 & x \\ 0 & 1 \end{psmallmatrix}^{-1} - \mathrm{Id}) = 0$, implying that $\mathrm{u}_x$, $\gamma\mathrm{u}_{x}^{-1} \in \PSL_2(\R)$ have each a real eigenspace, falling into the previous case. Writing $\gamma_1 := \gamma\mathrm{u}_{x}^{-1}$, we may then define $\widehat{\gamma} := \widehat{\gamma}_1 \widehat{\mathrm{u}}_{x}$. Notice that $x = x(\gamma)$ is the unique real number satisfying the $\det = 0$ condition, therefore $\widehat{\gamma} \in \SLt_2(\R)$ is uniquely determined.
  \end{enumerate}
 \end{rem}

%%%%%%%%%%%%%%%%%%%%%%%%%%%%%%%%%%%%%%%%%%%%%%%%%
\section{Invariants in \texorpdfstring{$\Hmot_{+}(\bbS^1)$}{Homeo\~{}\_+(S\^{}1)}}\label{sec4}
 Following Section 5 of Ghys \cite{ghy01}, define the \emph{translation number} of $\alpha \in \Hmot_{+}(\bbS^1)$ as the limit
 \begin{equation}
  \tau(\alpha) := \lim_{n\to +\infty} \frac{\alpha^{n}(x)-x}{n}, \label{trnm}
 \end{equation}
 where $x\in \R$ is arbitrary. The function $\tau$ is well-defined and is a \emph{homogeneous quasi-homomorphism};\footnote{Let $G$ be a group. A function $f:G\to \R$ is a \emph{quasi-homomorphism} if $\sup_{g,h\in G} |f(gh) - f(g) - f(h)|< \infty$ If, in addition, $f(g^k) = kf(g)$ for every $g\in G$, $k\in \Z$, then $f$ is said to be \emph{homogeneous}.} in particular, it is invariant under conjugation. The goal of this section is to describe two additional invariants associated to an element $\alpha \in \Hmot_{+}(\bbS^1)$:
 \begin{itemize}
  \item The \emph{direction type}: composed of \emph{translation magnitude} $|\tau(\alpha)| \in \R_{\geq 0}$ and \emph{translation direction} $\sigma(\alpha) \in \{-1, 0, 1\}$; and\smallskip

  \item The \emph{function} $\ell^{\sharp}: \Hmot_{+}(\bbS^1) \to \R^{\sharp}_{\geq [0]}$.
 \end{itemize}
 \noindent
 Both notions encode some way to measure ``displacement'' in $\Hmot_{+}(\bbS^1)$. The translation direction $\sigma$ is a refinement of the sign of $\tau$, so that the pair $(|\tau|, \sigma)$ carries more information than just $\tau$. This pair gives us the \emph{direction type}.

%%%%%%%%%%%%%%%%%%%%%%%%%%%
\subsection{Direction type}
 For $\alpha \in \Hmot_{+}(\bbS^1)$, we define $\sigma(\alpha)$ as follows:
 \begin{itemize}
  \item If $|\tau(\alpha)| > 0$, then $\sigma(\alpha) := \tau(\alpha)/|\tau(\alpha)|$;

  \item If $\alpha \neq \mathrm{id}_{\R}$, $\tau(\alpha)=0$ and $\tau(\mathcal{T}_{\eps}\,\alpha) > 0$ for every $\eps > 0$, then $\sigma(\alpha) := +1$;

  \item If $\alpha \neq \mathrm{id}_{\R}$, $\tau(\alpha)=0$ and $\tau(\mathcal{T}_{-\eps}\,\alpha) < 0$ for every $\eps > 0$, then $\sigma(\alpha) := -1$;

  \item If none of the above is satisfied, we define $\sigma(\alpha) := 0$.
 \end{itemize}

 \noindent
 The case where both $\tau(\mathcal{T}_{\eps}\,\alpha)>0$ and $\tau(\mathcal{T}_{-\eps}\,\alpha)< 0$ are true for all $\eps>0$ is only satisfied by the identity. Indeed, if $\alpha \neq \mathrm{id}_{\R}$, then there is $x\in \R$ such that either $\alpha(x) > x$ or $\alpha(x) < x$; thus, letting $\eps := |\alpha(x)-x|$, either $\tau(\mathcal{T}_{-\eps}\,\alpha) = 0$ or $\tau(\mathcal{T}_{\eps}\,\alpha) = 0$. Therefore, $\sigma$ is well-defined for all elements of $\Hmot_{+}(\bbS^1)$. 

 \begin{defn}[Direction types]\label{dirtype}
  Given $\alpha \in \Hmot_{+}(\bbS^1)$ with $\alpha \neq \id_{\R}$, we say that $\alpha$ has \emph{direction type} $(\frac{\tau}{|\tau|}, \sigma)$, where we define $\frac{\tau}{|\tau|} :=0$ if $\tau(\alpha)=0$.
  
  \noindent
  \begin{center}
  \renewcommand{\arraystretch}{1.5}
  \begin{tabular}{|c|c|c|}
   \hline
   \emph{\textbf{Direction type}} & \textnormal{\textbf{$(\frac{\tau}{|\tau|}, \sigma)$-pair}} & \textnormal{\textbf{Description}} \\ \hline\hline
   \emph{Forward} & $(1, 1)$ & {$\forall x\in\R,\, \alpha(x) - x > 0$} \\ \hline
   \emph{Semi-forward} & $(0, 1)$ & {$\alpha$ has a fixed point; $\forall x\in\R,\, \alpha(x) - x \geq 0$} \\ \hline
   \emph{Alternating} & $(0, 0)$ & {$\alpha(x) - x$ changes sign} \\ \hline
   \emph{Semi-backward} & $(0, -1)$ & {$\alpha$ has a fixed point; $\forall x\in\R,\, \alpha(x) - x \leq 0$} \\ \hline
   \emph{Backward} & $(1, -1)$ & {$\forall x\in\R,\, \alpha(x) - x < 0$} \\
   \hline %%%
  \end{tabular}
  \end{center}
  \noindent
  If $\alpha = \id_{\R}$, then we simply say it is the identity.
 \end{defn}

 If $\alpha$ is of alternating type, then so is $\alpha^{-1}$. If $\alpha$ is of forward (resp. semi-forward) type, then $\alpha^{-1}$ is of backward (resp. semi-backward) type; that is, $\frac{\tau}{|\tau|}(\alpha^{-1}) = \frac{\tau}{|\tau|}(\alpha)$ and $\sigma(\alpha^{-1}) = -\sigma(\alpha)$. Since $\tau$ is invariant under conjugation, so is $\frac{\tau}{|\tau|}$; the same holds for $\sigma$:

 \FloatBarrier
 \begin{figure}[!htb]
 \centering
 \begin{minipage}{0.32\textwidth}\begin{center}
  \begin{tikzpicture}[scale=0.7, every node/.style={scale=0.7}]
   \draw[help lines, black!10] (-1,-1) grid (5,5);
   \draw[help lines, ystep=2, xstep=2, black!20] (-1,-1) grid (5,5);
   \draw[black] (0,-0.02) node {\large{$\bullet$}};
   \draw[black] (4,3.98) node {\large{$\circ$}};
   \draw[black!80] (0,0)--(-.1,-.1) node[anchor=north east] {0};
   \draw[black!80] (0,2)--(-.1,2-.2) node[anchor=east] {.5};
   \draw[black!80] (0,4)--(-.1,4-.2) node[anchor=east] {1};
   \draw[black!80] (2,0)--(2-.2,-.1) node[anchor=north] {.5};
   \draw[black!80] (4,0)--(4-.2,-.1) node[anchor=north] {1};

   \draw[->, black] (0,-1)--(0,5);
   \draw[->, black] (-1,0)--(5,0);
   \draw[-, dashed, black] (-1,-1)--(5,5);

   \draw[scale=1,domain=-1:3.85,smooth,variable=\x,blue!50!black ,thick] plot ({\x},{\x + 4*sin(2*pi*(\x/4) r)/(2*pi) + 4/pi});
  \end{tikzpicture}

  {\small $\alpha(x) = x + \frac{\sin(2\pi x)}{2\pi} + \frac{1}{\pi}$}

  (Forward)
 \end{center}\end{minipage}
 \begin{minipage}{0.32\textwidth}\begin{center}
  \begin{tikzpicture}[scale=0.7, every node/.style={scale=0.7}]
   \draw[help lines, black!10] (-1,-1) grid (5,5);
   \draw[help lines, ystep=2, xstep=2, black!20] (-1,-1) grid (5,5);
   \draw[black] (0,-0.02) node {\large{$\bullet$}};
   \draw[black] (4,3.98) node {\large{$\circ$}};
   \draw[black!80] (0,0)--(-.1,-.1) node[anchor=north east] {0};
   \draw[black!80] (0,2)--(-.1,2-.2) node[anchor=east] {.5};
   \draw[black!80] (0,4)--(-.1,4-.2) node[anchor=east] {1};
   \draw[black!80] (2,0)--(2-.2,-.1) node[anchor=north] {.5};
   \draw[black!80] (4,0)--(4-.2,-.1) node[anchor=north] {1};

   \draw[->, black] (0,-1)--(0,5);
   \draw[->, black] (-1,0)--(5,0);
   \draw[-, thick, purple!25!black] (-1,-1)--(5,5);
  \end{tikzpicture}

  $\alpha(x) = x$

  \textit{(Identity)}
 \end{center}\end{minipage}
 \begin{minipage}{0.32\textwidth}\begin{center}
  \begin{tikzpicture}[scale=0.7, every node/.style={scale=0.7}]
   \draw[help lines, black!10] (-1,-1) grid (5,5);
   \draw[help lines, ystep=2, xstep=2, black!20] (-1,-1) grid (5,5);
   \draw[black] (0,-0.02) node {\large{$\bullet$}};
   \draw[black] (4,3.98) node {\large{$\circ$}};
   \draw[black!80] (0,0)--(-.1,-.1) node[anchor=north east] {0};
   \draw[black!80] (0,2)--(-.1,2-.2) node[anchor=east] {.5};
   \draw[black!80] (0,4)--(-.1,4-.2) node[anchor=east] {1};
   \draw[black!80] (2,0)--(2-.2,-.1) node[anchor=north] {.5};
   \draw[black!80] (4,0)--(4-.2,-.1) node[anchor=north] {1};

   \draw[->, black] (0,-1)--(0,5);
   \draw[->, black] (-1,0)--(5,0);
   \draw[-, dashed, black] (-1,-1)--(5,5);

   \draw[scale=1,domain=.15:5,smooth,variable=\x,red!50!black ,thick] plot ({\x},{\x + 4*sin(2*pi*(\x/4) r)/(2*pi) - 4/pi});
  \end{tikzpicture}

  {\small $\alpha(x) = x + \frac{\sin(2\pi x)}{2\pi} - \frac{1}{\pi}$}

  (Backward)
 \end{center}\end{minipage}
 %%%%%%%%%%%%%%%%%%%%%%%%%%%

 \bigskip
 \noindent
 \begin{minipage}{0.32\textwidth}\begin{center}
  \begin{tikzpicture}[scale=0.7, every node/.style={scale=0.7}]
   \draw[help lines, black!10] (-1,-1) grid (5,5);
   \draw[help lines, ystep=2, xstep=2, black!20] (-1,-1) grid (5,5);
   \draw[black] (0,-0.02) node {\large{$\bullet$}};
   \draw[black] (4,3.98) node {\large{$\circ$}};
   \draw[black!80] (0,0)--(-.1,-.1) node[anchor=north east] {0};
   \draw[black!80] (0,2)--(-.1,2-.2) node[anchor=east] {.5};
   \draw[black!80] (0,4)--(-.1,4-.2) node[anchor=east] {1};
   \draw[black!80] (2,0)--(2-.2,-.1) node[anchor=north] {.5};
   \draw[black!80] (4,0)--(4-.2,-.1) node[anchor=north] {1};

   \draw[->, black] (0,-1)--(0,5);
   \draw[->, black] (-1,0)--(5,0);
   \draw[-, dashed, black] (-1,-1)--(5,5);

   \draw[scale=1, domain={-.5-1/(pi)}:{4.5-1/(pi)}, smooth, variable=\x, blue!50!purple!50!black, thick] plot ({\x},{\x + 4*sin(2*pi*(\x/4) r)/(2*pi) + 4/(2*pi)});

   \draw[black] (3,2.98) node {\rotatebox{30}{\Huge{$\mathbf{\times}$}}};
  \end{tikzpicture}

  {\small $\alpha(x) = x + \frac{\sin(2\pi x)}{2\pi} + \frac{1}{2\pi}$}

  (Semi-forward)
 \end{center}\end{minipage}
 \begin{minipage}{0.32\textwidth}\begin{center}
  \begin{tikzpicture}[scale=0.7, every node/.style={scale=0.7}]
   \draw[help lines, black!10] (-1,-1) grid (5,5);
   \draw[help lines, ystep=2, xstep=2, black!20] (-1,-1) grid (5,5);
   \draw[black] (0,-0.02) node {\large{$\bullet$}};
   \draw[black] (4,3.98) node {\large{$\circ$}};
   \draw[black!80] (0,0)--(-.1,-.1) node[anchor=north east] {0};
   \draw[black!80] (0,2)--(-.1,2-.2) node[anchor=east] {.5};
   \draw[black!80] (0,4)--(-.1,4-.2) node[anchor=east] {1};
   \draw[black!80] (2,0)--(2-.2,-.1) node[anchor=north] {.5};
   \draw[black!80] (4,0)--(4-.2,-.1) node[anchor=north] {1};

   \draw[->, black] (0,-1)--(0,5);
   \draw[->, black] (-1,0)--(5,0);
   \draw[-, dashed, black] (-1,-1)--(5,5);

   \draw[scale=1,domain=-.5:4.5,smooth,variable=\x,purple!50!black,thick] plot ({\x},{\x + 4*sin(2*pi*(\x/4) r)/(2*pi)});

   \draw[black] (0,-0.02) node {\rotatebox{30}{\Huge{$\mathbf{\times}$}}};
   \draw[black] (2,1.98) node {\rotatebox{30}{\Huge{$\mathbf{\times}$}}};
   \draw[black] (4,3.98) node {\rotatebox{30}{\Huge{$\mathbf{\times}$}}};
  \end{tikzpicture}

  $\alpha(x) = x + \frac{\sin(2\pi x)}{2\pi}$

  (Alternating)
 \end{center}\end{minipage}
 \begin{minipage}{0.32\textwidth}\begin{center}
  \begin{tikzpicture}[scale=0.7, every node/.style={scale=0.7}]
   \draw[help lines, black!10] (-1,-1) grid (5,5);
   \draw[help lines, ystep=2, xstep=2, black!20] (-1,-1) grid (5,5);
   \draw[black] (0,-0.02) node {\large{$\bullet$}};
   \draw[black] (4,3.98) node {\large{$\circ$}};
   \draw[black!80] (0,0)--(-.1,-.1) node[anchor=north east] {0};
   \draw[black!80] (0,2)--(-.1,2-.2) node[anchor=east] {.5};
   \draw[black!80] (0,4)--(-.1,4-.2) node[anchor=east] {1};
   \draw[black!80] (2,0)--(2-.2,-.1) node[anchor=north] {.5};
   \draw[black!80] (4,0)--(4-.2,-.1) node[anchor=north] {1};

   \draw[->, black] (0,-1)--(0,5);
   \draw[->, black] (-1,0)--(5,0);
   \draw[-, dashed, black] (-1,-1)--(5,5);

   \draw[scale=1,domain={-.5+1/(pi)}:{4.5+1/(pi)},smooth,variable=\x,red!50!purple!50!black,thick] plot ({\x},{\x + 4*sin(2*pi*(\x/4) r)/(2*pi) - 4/(2*pi)});

   \draw[black] (1,.98) node {\rotatebox{30}{\Huge{$\mathbf{\times}$}}};
  \end{tikzpicture}

  {\small $\alpha(x) = x + \frac{\sin(2\pi x)}{2\pi} - \frac{1}{2\pi}$}

  (Semi-backward)
 \end{center}\end{minipage}
  \captionsetup{singlelinecheck=off}
  \caption{Examples of functions of each direction type.}
 \label{figDT}
 \end{figure}

 \begin{lem}\label{dirinv}
  Direction type is invariant under conjugation in $\Hmot_{+}(\bbS^1)$.
 \end{lem}
 \begin{proof}
  Let $\alpha,\beta \in \Hmot_{+}(\bbS^1)$. If $\alpha$ is of forward type, then $\alpha(x) > x$ for all $x\in \R$, which implies $\alpha(\beta^{-1}(x)) > \beta^{-1}(x)$; thus, as $\beta\beta^{-1}(x) = x$, it follows that $\beta\alpha\beta^{-1}(x) > x$. The same holds for semi-forward type, and the case where $\alpha$ is of backward or semi-backward type goes analogously.
  
  The set of elements in $\Hmot_{+}(\bbS^1)$ having a fixed point in $\R$ is clearly invariant under conjugation. By Definition \ref{dirtype}, it decomposes as the disjoint union of semi-forward, semi-backward, and alternating. since the first two cases are invariant under conjugation, so must be the third one.
 \end{proof}

 %The specific value of $\tau(\alpha)$ will not be too important for us during the remainder of this paper. Instead, we shall make more use of the \emph{length$^{\sharp}$} function as a measure of how much $\alpha$ ``moves''. Hence, the two invariants mentioned at the beginning of this section (namely, $|\tau|$ and $\sigma$) are more accurately conceptualized in our context as simply the parts that constitute the direction type.

%%%%%%%%%%
\subsection{Displacement length}\label{lsharp}
 As defined in Lemma 3.7 of Zhang \cite{zhan01}, consider now the \emph{length} function:
 \begin{equation}
  \ell:\Hmot_{+}(\bbS^1) \to \R_{\geq 0}, \qquad \ell(\alpha) := \sup_{t\in [0,1]} |\alpha(t)-t|. \label{lngth}
 \end{equation}
 Since $\alpha(t+1) = \alpha(t)+1$ for all $t\in\R$ and $[0,1]$ is compact, $\ell$ is well-defined, and we have $\ell(\alpha) = \max_{t\in[x, x+1]} |\alpha(t)-t|$ for all $x\in\R$, with $\ell(\alpha) = 0$ if and only if $\alpha = \id_{\R}$. One checks that $\ell$ is $1$-Lipschitz in $\Hmot_{+}(\bbS^1)$ natural metric, and hence continuous. Moreover, for any translation $\mathcal{T}_y$ with $y\in\R$ it holds that $\ell(\mathcal{T}_y) = |y|$.

 \begin{lem}[Basic properties of $\ell$]\label{bsinq}
  Let $\alpha,\beta \in \Hmot_{+}(\bbS^1)$. The following hold:
  \begin{enumerate}[label=\textnormal{(\roman*)}]
   \item $|\tau(\alpha)| \leq \ell(\alpha)$;

   \item $\ell(\alpha\beta) \leq \ell(\alpha) + \ell(\beta)$;

   \item If $\alpha$ has a fixed point, then $\ell(\alpha) < 1$;

   \item $\ell([\alpha,\beta]) < 2$.
  \end{enumerate}
 \end{lem}
 \begin{proof}
  For part (i), for each $n\geq 1$ we have
  \begin{align*}
   |\alpha^{n}(x) - x| &\leq |\alpha(\alpha^{n-1}(x)) -\alpha^{n-1}(x)| + |\alpha^{n-1}(x) - x| \\
   &\leq \ell(\alpha) + |\alpha^{n-1}(x) - x|,
  \end{align*}
  and thus, $|\alpha^{n}(x) -x| \leq n\,\ell(\alpha)$, which implies $|\tau(\alpha)| \leq \ell(\alpha)$.
  
  For part (ii), since $\beta$ is bijective and $\beta(1) = \beta(0) + 1$, we have
  \begin{align*}
   \ell(\alpha\beta) &= \max_{t\in [0,1]} |\alpha\beta(t)-t| \\
   &\leq \max_{t\in [0,1]} \big(|\alpha\beta(t)-\beta(t)| + |\beta(t)-t|\big) \\
   &\leq \max_{u\in [\beta(0),\beta(1)]} |\alpha(u)-u| + \max_{t\in [0,1]} |\beta(t)-t| = \ell(\alpha) + \ell(\beta).
  \end{align*}

  For part (iii), let $x^{*}\in\R$ be some fixed point of $\alpha$. Since $\ell(\alpha) = \sup_{t\in[x^{*},x^{*}+1]}|\alpha(t)-t|$, take $t^{*}\in [x^{*}, x^{*}+1]$ such that $\ell(\alpha) = |\alpha(t^{*})-t^{*}|$. From the fact that $\alpha$ is strictly increasing and $\alpha(x^{*}+1) = \alpha(x^{*})+1$, we deduce that $\alpha([x^{*},x^{*}+1]) = [x^{*},x^{*}+1]$. Therefore, we have $T^{*} := \alpha(t^{*}) \in (x^{*},x^{*}+1)$, the endpoints of the interval being excluded because $\alpha(x^{*}) = x^{*}$, from where it follows that $\ell(\alpha) = |T^{*}-t^{*}| < 1$.

  For part (iv), by Lemma \ref{auttcent}, there are $\varphi\in \Hmo_{+}(\bbS^1)$ and $n\in\Z$ such that $\alpha = \mathcal{T}_n \widetilde{\varphi}$, where $\widetilde{\varphi}$ is the canonical lifting of $\varphi$. Since $\mathcal{T}_n$ is an element of the center, it follows that $[\alpha,\beta] = [\widetilde{\varphi}, \beta] = \widetilde{\varphi} \beta \widetilde{\varphi}^{-1} \beta^{-1}$. From Lemma \ref{dirinv}, we know that $\widetilde{\varphi}^{-1}$ and $\beta \widetilde{\varphi}^{-1} \beta^{-1}$ have the same direction type, therefore, either both $\widetilde{\varphi}$ and $\beta \widetilde{\varphi}^{-1} \beta^{-1}$ have fixed points or they have opposite directions (i.e., $\sigma(\widetilde{\varphi}) = - \sigma(\beta \widetilde{\varphi}^{-1} \beta^{-1})$). We consider these two cases separately:
  \begin{description}
   \item[Case 1]
    If both $\widetilde{\varphi}$ and $\beta \widetilde{\varphi}^{-1} \beta^{-1}$ have fixed points, then by item (iii) we have $\ell(\widetilde{\varphi}) < 1$ and $\ell(\beta \widetilde{\varphi}^{-1} \beta^{-1}) < 1$. Thus, by item (ii), we conclude that $\ell([\widetilde{\varphi}, \beta]) < 2$. \smallskip
   
   \item[Case 2]
    Suppose that $\sigma(\widetilde{\varphi})=1$. By the construction of canonical liftings, we have $0\leq \widetilde{\varphi}(0) < 1$, hence, $\max_{t\in [0,1]} |\widetilde{\varphi}(t)-t| \leq |\widetilde{\varphi}(1) - 0| < 2$. Since $\beta$ is a strictly increasing self-homeomorphism of $\R$, it holds that
    \begin{align*}
     &\phantom{\implies} &&\phantom{\beta\widetilde{\varphi}^{-1}\beta^{-1}(u)}\negphantom{$\widetilde{\varphi}^{-1}(t)$} \widetilde{\varphi}^{-1}(t) > t - 2\phantom{\beta(t) - 2}\negphantom{$t - 2$} \quad (\text{applying } \beta \text{ on both sides})\\
     &\implies &&\phantom{\beta\widetilde{\varphi}^{-1}\beta^{-1}(u)}\negphantom{$\beta\widetilde{\varphi}^{-1}(t)$} \beta\widetilde{\varphi}^{-1}(t) > \beta(t) - 2 \quad (\text{substituting } t \,\mapsto\, \beta^{-1}(u))\\
     &\implies &&\beta\widetilde{\varphi}^{-1}\beta^{-1}(u) > u - 2,
    \end{align*}
    which implies that $\ell(\beta\widetilde{\varphi}^{-1}\beta^{-1}) < 2$ as well. Hence, since $\sigma(\beta\widetilde{\varphi}^{-1}\beta^{-1})=-1$, for all $t\in\R$ we have
    \begin{gather*}
     0\leq \widetilde{\varphi}\big(\beta\widetilde{\varphi}^{-1}\beta^{-1}(t)\big) - \beta\widetilde{\varphi}^{-1}\beta^{-1}(t) < 2, \\
     -2 < \beta\widetilde{\varphi}^{-1}\beta^{-1}(t) - t \leq 0,
    \end{gather*}
    and thus
    \[ -2 < \widetilde{\varphi}\beta\widetilde{\varphi}^{-1}\beta^{-1}(t) - t < 2, \]
    which is equivalent to $\ell([\widetilde{\varphi}, \beta]) < 2$, as required. The case for when $\sigma(\widetilde{\varphi})=-1$ goes analogously. \qedhere
  \end{description}
 \end{proof}

 Conceptually, we would like to have a conjugacy invariant version of $\ell$. Based on the ``$\delta^{\sup}$'' function defined in Section 3 of Mochizuki \cite{moc16}, let
 \begin{equation*}
  \begin{aligned}
   \R^{\sharp} :=& ~\bigslant{\R}{\cong}, \text{ where } x \cong y \iff
    \begin{cases}
     x = y \in \Z, \text{ or} \\
     x,y \in (k,k+1) \text{ for some } k\in\Z.
    \end{cases} \\
   =& ~\{\ldots, [(-2,-1)], [-1], [(-1,0)], [0], [(0,1)], [1], [(1,2)], \ldots \},
  \end{aligned}
 \end{equation*}
 and denote the quotient map by $\R \ni x \,\mapsto\, x^{\sharp} \in \R^{\sharp}$. This object also arises as the quotient of $\R$ through the action of the subgroup of $\Hmot_{+}(\bbS^1)$ which fixes $\Z$, i.e., $\mathrm{Stab}_{\Hmot_{+}(\bbS^1)}(0) \curvearrowright \R$, where
 \begin{equation*}
  \mathrm{Stab}_{\Hmot_{+}(\bbS^1)}(0) := \{\alpha \in \Hmot_{+}(\bbS^1) ~|~ \alpha(0) = 0 \}.
 \end{equation*}
 Despite the natural order isomorphism between $\R^{\sharp}$ with the half-integers $\frac{1}{2}\Z$, the group structure of the latter is not shared by the former, so we avoid the identification. An alternative description of $x^{\sharp}$ would be the set $\{\lfloor x\rfloor, \lceil x\rceil\}$, highlighting the precise loss of information that takes place with the mapping $x\mapsto x^{\sharp}$ -- this is called \emph{indeterminacies} in Section 3 of Mochizuki \cite{moc16}.
 
 The main feature of $\R^{\sharp}$ is that, defining
 \begin{equation}
  \ell^{\sharp}(\alpha) := \big(\ell(\alpha)\big)^{\sharp} \in \R^{\sharp},\label{deflsharp}
 \end{equation}
 we have

 \begin{prop}\label{lcinv}
  $\ell^{\sharp}$ is invariant under conjugation in $\Hmot_{+}(\bbS^1)$.
 \end{prop}
 \begin{proof}
  Let $\alpha,\beta \in \Hmot_{+}(\bbS^1)$. If $\alpha \neq \id_{\bbS^1}$ has a fixed point, then by Lemma \ref{bsinq} (iii) we have $\ell^{\sharp}(\alpha) = [(0,1)]$, and by Lemma \ref{dirinv} we have that $\beta\alpha\beta^{-1}$ has a fixed point as well, implying that $\ell^{\sharp}(\beta\alpha\beta^{-1}) = \ell^{\sharp}(\alpha) = [(0,1)]$. Assume, then, that $\alpha$ is of forward type. We are going to show that
  \begin{gather*}
   \ell(\alpha) < n \implies \ell(\beta\alpha\beta^{-1}) < n, \quad \ell(\alpha) \leq n \implies \ell(\beta\alpha\beta^{-1}) \leq n, \\
   \ell(\alpha) > n \implies \ell(\beta\alpha\beta^{-1}) > n \quad\text{and}\quad \ell(\alpha) \geq n \implies \ell(\beta\alpha\beta^{-1}) \geq n.
  \end{gather*}
  
  If $\ell(\alpha) < n$, suppose that $\ell(\beta\alpha\beta^{-1}) \geq n$; i.e., that there is $x\in\R$ such that
  \[ \beta \alpha \beta^{-1}(x) \geq x + n. \]
  Since $\beta^{-1}$ is increasing and commutes with $\mathcal{T}_n$, we have $\alpha(\beta^{-1}(x)) \geq \beta^{-1}(x) + n$, which implies that there is $y\in \R$ for which $\alpha(y) - y \geq n$, leading to contradiction. The case for when $\ell(\alpha) \leq n$ is similar.

  If $\ell(\alpha) \geq n$, suppose that $\ell(\beta\alpha\beta^{-1}) < n$; i.e., that for every $x\in\R$ we have
  \[ \beta \alpha \beta^{-1}(x) < x + n. \]
  Since $\beta^{-1}$ is increasing and commutes with $\mathcal{T}_n$, we have $\alpha(\beta^{-1}(x)) < \beta^{-1}(x) + n$, which implies that for every $y\in \R$ we have $\alpha(y) - y < n$, leading to contradiction. The same follows if $\ell(\alpha) > n$, and the case for when $\alpha$ is of backward type goes analogously, concluding the proof.
 \end{proof}

%%%%%%%%%
\section{Conjugacy classes of \texorpdfstring{$\SLt_2(\R)$}{SL\~{}(2,R)}}\label{secsltt}
 Write $I := \Cl(\PSL_2(\R))$, and for each $i\in I$, choose and fix a representative element $\varphi_i \in \PSL_2(\R)$, so that $\Cl(\PSL_2(\R)) = \{[\varphi_i] ~|~ i\in I\}$. By \eqref{gtilde}, the canonical liftings $\widetilde{\varphi}_i$ of $\varphi_i$ are in $\SLt_2(\R)$. Together with the center, these are sufficient to describe $\Cl(\SLt_2(\R))$.

 \begin{prop}\label{nvpi}
  For every $\alpha \in \SLt_2(\R)$, there is a unique $i\in I$ (as defined above) and a unique $k\in\Z$ for which $\alpha \sim \mathcal{T}_{k}\, \widetilde{\varphi}_i$.
 \end{prop}
 \begin{proof}
  Let $\alpha\in \SLt_2(\R)$, and take $\varphi_i \in \PSL_2(\R)$ as defined above such that $\overline{\alpha} = p \varphi_{i} p^{-1}$, for some $p\in\PSL_2(\R)$. Lifting to $\SLt_2(\R)$, by \eqref{gtilde} we have that $\alpha = \widetilde{p}\, \mathcal{T}_k \widetilde{\varphi}_i\, \widetilde{p}^{-1}$ for some $k\in \Z$. Since $\mathcal{T}_k\, \widetilde{\varphi}_i \sim \mathcal{T}_{\ell}\, \widetilde{\varphi}_j$ implies $i=j$, all we need is to show that $k=\ell$.
  
  Let $\beta\in \SLt_2(\R)$ be such that $\mathcal{T}_{k-\ell}\, \widetilde{\varphi}_i = \beta \widetilde{\varphi}_i \beta^{-1}$. As $\mathcal{T}_{k-\ell}$ is central in $\SLt_2(\R)$, for every $n\geq 1$, we have
  \[ \mathcal{T}_{n(k-\ell)} = [\beta,\widetilde{\varphi}_i^n]. \]
  However, by Lemma \ref{bsinq} (iv), we have $n(k-\ell) = \ell([\beta,\widetilde{\varphi}_i^n]) < 2$ for every $n\geq 1$. Therefore, $k=\ell$.
 \end{proof}

 \begin{rem}[Torsor structure on conjugacy classes]
  Consider the exact sequence representing a central extension $E$ of a group $G$:
  \begin{equation*}
   \begin{tikzcd}
    1\arrow[r] & Z\arrow[r] & E\arrow[r, "\phi"] & G\arrow[r] & 1.
   \end{tikzcd}
  \end{equation*}
  Proposition \ref{nvpi} shows that for $Z = \Z$, $E = \SLt_2(\R)$, $G = \PSL_2(\R)$, the sequence above has the the following property:
  
  \begin{center}
   \textit{$\Cl(E)$ is a $Z$-torsor over $\Cl(G)$ along the induced map $\phi^{\Cl}: \Cl(E) \twoheadrightarrow \Cl(G)$.}
  \end{center}
  This means that, if $pq = qp$ in $G$, then $\widetilde{p}\widetilde{q} = \widetilde{q}\widetilde{p}$ for any liftings $\widetilde{p}$, $\widetilde{q} \in E$. Indeed, since $\widetilde{(pq)} = \widetilde{(qp)}$, we have $\widetilde{(pq)} = \mathcal{T}_{n}\widetilde{p}\widetilde{q}$ and $\widetilde{(qp)} = \mathcal{T}_{k}\widetilde{q}\widetilde{p}$ for some $k,n\in\Z$, and so $\mathcal{T}_{n}\widetilde{q} \sim \mathcal{T}_{k}\widetilde{q}$. By Proposition \ref{nvpi}, $n=k$, and thus $\widetilde{p}\widetilde{q} = \widetilde{q}\widetilde{p}$.
  
  Other group extensions of interest that satisfy this property include:
  \begin{itemize}
   \item The universal covering of real symplectic groups $\widetilde{Sp}_{2n}(\R) \twoheadrightarrow \mathrm{Sp}_{2n}(\R)$ (cf. Proposition 2.3 (i) of Barge--Ghys \cite{barghy92});\smallskip
   
   \item The universal covering $\Hmot_0(\mathbb{T}^n) \twoheadrightarrow \Hmo_0(\mathbb{T}^n)$, of the identity component of the self-homeomorphisms group of the $n$-torus (cf. Remarque 2.12, p. 22 of Herman \cite{her79}).\smallskip
  \end{itemize}
  As a non-example, the classical double cover $\mathrm{SU}(2) \to \mathrm{SO}(3)$ does not satisfy this property.
 \end{rem}

 \begin{thm}\label{finale}
  In the notation of Subsection \ref{introtbl}, the conjugacy classes of $\SLt_2(\R)$ are described by the following table, with each $\vartheta \in (0,\pi)$, $\lambda \in \R_{>1}$, and $n\in \Z_{\geq 1}$ denoting a different conjugacy class:
  \noindent
  \begin{center}
   \renewcommand{\arraystretch}{1.5}
   \begin{tabular}{|c|c|}
    \hline
    \textnormal{\textbf{Conjugacy class}} & \textnormal{$\mathbf{\ell^{\sharp}}$} \\
    \hline
    $[\id_{\R}]$ & $[0]$ \\
    \hline
    $[\widetilde{\mathrm{u}}_1]$, $[\widetilde{\mathrm{u}}_{-1}]$, $[\widetilde{\mathrm{a}}_{\lambda}]$, $[\widetilde{\rho}_{\vartheta}]$, $[\mathcal{T}_{-1}\,\widetilde{\rho}_{\vartheta}]$ & $[(0,1)]$ \\
    \hline
    $[\mathcal{T}_{n}]$, $[\mathcal{T}_{-n}]$, $[\mathcal{T}_{n}\,\widetilde{\mathrm{u}}_1]$, $[\mathcal{T}_{-n}\, \widetilde{\mathrm{u}}_{-1}]$ & $[n]$ \\
    \hline
    $[\mathcal{T}_{-n}\,\widetilde{\mathrm{u}}_1]$, $[\mathcal{T}_{n}\,\widetilde{\mathrm{u}}_{-1}]$, $[\mathcal{T}_{n}\, \widetilde{\mathrm{a}}_{\lambda}]$ & \multirow{2}{*}{$[(n,n+1)]$} \\
    $[\mathcal{T}_{-n}\, \widetilde{\mathrm{a}}_{\lambda}]$, $[\mathcal{T}_{n}\,\widetilde{\rho}_{\vartheta}]$, $[\mathcal{T}_{-1-n}\,\widetilde{\rho}_{\vartheta}]$ & \\
    \hline
   \end{tabular}
  \end{center}
 \end{thm}
 \begin{proof}
  By Proposition \ref{nvpi}, the conjugacy classes listed above are all inequivalent and represent all of $\SLt_2(\R)$. It only remains for us to check the value of $\ell^{\sharp}$ of each class, which we know by Proposition \ref{lcinv} to be conjugacy invariant.
  
  Clearly, $\ell^{\sharp}(\mathrm{id}_{\R}) = [0]$ and $\ell^{\sharp}(\mathcal{T}_n) = \ell^{\sharp}(\mathcal{T}_{-n}) = [n]$. Furthermore, since $\widetilde{\rho}_{\vartheta} = \mathcal{T}_{\vartheta/\pi}$, for every $k\in \Z_{\geq 0}$ we have $\ell^{\sharp}(\mathcal{T}_{k}\,\widetilde{\rho}_{\vartheta}) = \ell^{\sharp}(\mathcal{T}_{-1-k}\,\widetilde{\rho}_{\vartheta}) = [(k,k+1)]$.
  Since $\widetilde{\mathrm{u}}_{1}$, $\widetilde{\mathrm{u}}_{-1}$ and $\widetilde{\mathrm{a}}_{\lambda}$ have fixed points, by Lemma \ref{bsinq} (iii) it holds that $\ell^{\sharp}(\widetilde{\mathrm{u}}_1) = \ell^{\sharp}(\widetilde{\mathrm{u}}_{-1}) = \ell^{\sharp}(\widetilde{\mathrm{a}}_{\lambda}) = [(0,1)]$. From the definition of the action $\PSL_2(\R) \curvearrowright \bbS^1$, one verifies that $\widetilde{\mathrm{u}}_{-1}$ is of semi-forward direction type, while $\widetilde{\mathrm{u}}_{1}$ is of semi-backward type. From Lemma \ref{bsinq} (iii), we know that $\ell(\widetilde{\mathrm{u}}_{-1}),\ell(\widetilde{\mathrm{u}}_{1}) < 1$, so it follows that, for every $n\geq 1$,
  \begin{gather*}
   \ell^{\sharp}(\mathcal{T}_{n}\,\widetilde{\mathrm{u}}_{1}) = \ell^{\sharp}(\mathcal{T}_{-n}\,\widetilde{\mathrm{u}}_{-1}) = [n], \\
   \ell^{\sharp}(\mathcal{T}_{-n}\,\widetilde{\mathrm{u}}_{1}) = \ell^{\sharp}(\mathcal{T}_{n}\,\widetilde{\mathrm{u}}_{-1}) = [(n,n+1)].
  \end{gather*}
  Finally, for $\widetilde{\mathrm{a}}_{\lambda}$, since $\ell(\widetilde{\mathrm{a}}_{\lambda}) < 1$ and $\widetilde{\mathrm{a}}_{\lambda}$ is of alternating type, both $\ell^{\sharp}(\mathcal{T}_{n}\,\widetilde{\mathrm{a}}_{\lambda})$ and $\ell^{\sharp}(\mathcal{T}_{-n}\,\widetilde{\mathrm{a}}_{\lambda})$ must equal $[(n,n+1)]$, completing the proof.
 \end{proof}
 
%%%%%%%%%%%%%%%%%%%%%%%
\addtocontents{toc}{\protect\setcounter{tocdepth}{0}}
\section*{Acknowledgements}
 I would like to thank Shinichi Mochizuki and Go Yamashita for having suggested me this topic, as well as for their patient guidance with early versions of this research note. I also thank the anonymous referee for their valuable input on the presentation of this article.
 
\addtocontents{toc}{\protect\setcounter{tocdepth}{1}}
%%%%%%%%%%%%%%%%%%%%%%%

% ----------------------------------------------------------------
\bibliographystyle{amsplain}
\bibliography{$HOME/Academie/Recherche/_latex/bibliotheca}%
\end{document}